\documentclass[11pt]{amsart}

\usepackage{pdfsync}

\usepackage{amssymb,graphics,epsfig,overpic,amscd,amsmath}
\usepackage[usenames,dvipsnames]{color}
\usepackage[colorlinks=true,linkcolor=blue,citecolor=BrickRed]{hyperref}

\newtheorem{thm}{Theorem}[section]
\newtheorem{lem}[thm]{Lemma}

\theoremstyle{definition}

\newtheorem{note}[thm]{Note}

\newcommand{\R}{\mathbf{R}}

\newcommand{\ol}{\overline}
\newcommand{\C}{\mathcal{C}}
\newcommand{\Mobius}{M\"{o}bius }

\renewcommand{\(}{\left(}
\renewcommand{\)}{\right)}

\renewcommand{\S}{\mathbf{S}}

\renewcommand{\(}{\left(}
\renewcommand{\)}{\right)}
\renewcommand{\tilde}{\widetilde}
\renewcommand{\epsilon}{\varepsilon}
\renewcommand{\setminus}{\smallsetminus}

\DeclareMathOperator{\cl}{cl}

\DeclareMathOperator{\graph}{graph}
\DeclareMathOperator{\Hess}{Hess}
\DeclareMathOperator{\rank}{rank}

\title[Normal curvatures of asymptotically constant graphs]{Normal curvatures of asymptotically constant graphs and Carath\'{e}odory's conjecture}

\author{Mohammad Ghomi}
\address{School of Mathematics, Georgia Institute of Technology,
Atlanta, GA 30332}
\email{ghomi@math.gatech.edu}
\urladdr{www.math.gatech.edu/$\sim$ghomi}

\author{Ralph Howard}
\address{Department of Mathematics,
University of South Carolina,
Columbia, SC 29208}
\email{howard@math.sc.edu}
\urladdr{www.math.sc.edu/$\sim$howard}

\date{\today \,(Last Typeset)}
\thanks{The research of the first named author was supported in part by NSF grant DMS-0806305.}

\begin{document}

\begin{abstract}
We show that Carath\'{e}odory's conjecture, on umbilical points of closed convex  surfaces, may be reformulated in terms of the existence of at least one umbilic in the graphs of functions $f\colon\R^2\to\R$ whose gradient decays uniformly faster than $1/r$.  The divergence theorem then yields a pair of integral equations for the normal curvatures of these graphs, which establish some weaker forms of the conjecture. In particular, we show that  there are uncountably many principal lines in the graph of $f$ whose projection into $\R^2$ are parallel to any given direction.
\end{abstract}

\maketitle

\section{Introduction}
Carath\'{e}odory's celebrated conjecture \cite{berger:panorama,yau:problems1} which has been the subject of numerous investigations since 1920's \cite{struik}, asserts that every (sufficiently smooth) closed convex surface $M$ in Euclidean space $\R^3$ has at least two \emph{umbilics}, i.e., points where the principal curvatures of $M$ are equal. Almost all attempts to prove this claim have been concerned with establishing the more general  local conjecture of Loewner on the index of the singularities of principal line fields. Here, by contrast, we develop a global approach to this problem. To state our main results, let us say that a function  $f\colon\R^2\to\R$  is \emph{asymptotically constant} (at a sufficiently fast and uniform rate)
 provided that, in polar coordinates, 
 $$
 \lim_{r\to\infty}f(r,\theta)=const.,\quad\text{and}\quad \lim_{r\to\infty}r\,\|\nabla f(r,\theta)\|=0,
 $$
 uniformly with respect to $\theta$. In other words $f$ converges uniformly to a constant at infinity, while the norm of its gradient decays faster than $1/r$ uniformly in all directions.  
\begin{thm}\label{thm:1}
For every $\C^2$ closed convex surface $M\subset\R^3$ and umbilical point $p\in M$, there exists an  asymptotically constant $\C^2$ function $f\colon \R^2\to\R$, and a diffeomorphism between $M\setminus\{p\}$ and  $\graph(f)$ which preserves the principal directions. In particular, $\graph(f)$ has one fewer umbilic than $M$ has.
\end{thm}

Thus Carath\'{e}odory's conjecture may be reformulated in terms of existence of at least one umbilical point in the graphs of asymptotically constant  functions $f\colon\R^2\to\R$. We use this setting to obtain some evidence in support of the conjecture as stated in the next two results. For any point $p$ in the plane $\R^2$ and direction (or unit vector) $X$ in the circle $\S^1$, let $k(p,X)$ be the \emph{normal curvature} of the graph of $f$ at  the point $(p,f(p))$ in the direction of a  tangent vector  whose projection into $\R^2$ is parallel to $X$, i.e., set
\begin{equation}\label{eq:k}
k(p,X):=\frac{f_{XX}(p)}{\left(1+f_X^2(p)\right)\sqrt{1+\|\nabla f(p)\|^2}},
\end{equation}
where $f_X:=\langle\nabla f,X\rangle$ is the  derivative of $f$ in the direction $X$,  $f_{XX}:=(f_X)_X$, and $f_X^2$ denotes $(f_X)^2$. 
 Note that $f$ has an umbilic at $p$ (or $(p,f(p))$ is an umbilical point of $\graph(f)$) if  $k(p,\cdot)$ is constant on $\S^1$. So, if Carath\'{e}odory's conjecture were true, then for every pair of directions $X$, $Y$ there would exist a point $p$ such that $k(p,X)=k(p,Y)$. The next result shows that there are always an abundant supply of such points for any given pair of directions. Here $B_r$ denotes a closed ball of radius $r$ centered at the origin of $\R^2$, and 
 $$
 d\mathcal{A}_p:=\sqrt{1+\|\nabla f(p)\|^2}\,dx\wedge dy
 $$
is the area element of the graph of $f$.

\begin{thm}\label{thm:2}
Let $f\colon \R^2\to\R$ be an  asymptotically constant $\C^2$ function. Then for every pair of  directions $X$, $Y\in\S^1$, 
$$
 \lim_{r\to\infty}\int_{p\in B_r}\big(k(p,X)-k(p,Y)\big)\(1+f_X^2(p)\)\(1+f_Y^2(p)\)\,d\mathcal{A}_p=0.
$$
 In particular  the function $k(\,\cdot\,,X)-k(\,\cdot\,,Y)$ is either identically zero or else changes sign on  $\R^2$.
\end{thm} 

To describe our next result, let us say that $X\in\S^1$ is a principal direction of $f$ at $p\in\R^2$ provided that $X$ is parallel to the projection of a principal direction of  $\graph(f)$  at $(p,f(p))$ into $\R^2$. Next set
$X(\theta):=(\cos(\theta),\sin(\theta))$, and note that $X(\theta_0)$ is a principal direction of $f$ if and only if $\theta_0$ is a critical point of the function $\theta\mapsto k(p,X(\theta))$, for some $p$. Thus if Carath\'{e}odory's conjecture holds, then for every direction $X(\theta_0)$ there should exist a point $p$ such that $\frac{\partial k}{\partial\theta}(p,X(\theta_0))=0$. 
Again, we show that there is no shortage of such points:

\begin{thm}\label{thm:3}
Let $f\colon \R^2\to\R$ be an  asymptotically constant $\C^2$ function. Then for every direction  $X(\theta_0)\in\S^1$
$$
\lim_{r\to\infty}\int_{p\in B_r}\frac{\partial k}{\partial\theta}\big(p,X(\theta_0)\big)\(1+f_{X(\theta_0)}^2(p)\)\,d\mathcal{A}_p=0.
$$
In particular, the set of points $p\in\R^2$ where $X(\theta_0)$ is a principal direction of $f$ at $p$ is either  all of $\R^2$ or else separates $\R^2$.
\end{thm}

Note that  each of the last two results gives  considerably more information than can be subsumed  by a  positive resolution of Carath\'{e}odory's conjecture. The last theorem,   for instance, guarantees the existence of uncountably many principal lines in the graph of $f$ whose projection into $\R^2$ are parallel to any given direction, while the existence of an umbilic in the graph of $f$  ensures only the existence of  one such line for each direction. Also note that, to prove Carath\'{e}odory's conjecture, it suffices to show that the zero sets of $k(\,\cdot\,,X)-k(\,\cdot\,,Y)$ and $\frac{\partial k}{\partial\theta}(\cdot,X)$, which are abundant by Theorems \ref{thm:2} and \ref{thm:3}, have nonempty intersection for some pair of linearly independent directions $X$, $Y\in\S^1$; see Note \ref{note:pdes} for more on this approach.
Theorem \ref{thm:1} is proved using a combination of \Mobius inversions and the operation of moving a surface parallel to itself (Section \ref{sec:1}), while Theorems \ref{thm:2} and \ref{thm:3} are fairly quick applications of the divergence theorem after some computations (Section \ref{sec:2}). We will also discuss how to construct closed surfaces, including some convex ones, which are smooth and umbilic free in the complement of one point, and may be arbitrarily close to a sphere (Section \ref{sec:3}).

According to Struik \cite{struik}, the earliest references to the conjecture attributed to Carath\'{e}odory appear in the works of Cohn-Vossen, Blaschke, and Hamburger dating back to 1922.
The first significant results on the conjecture were due to Hamburger who established the analytic case in a series of long papers \cite{hamburgerI, hamburgerII, hamburgerIII} published in 1940--41. Attempts to  find shorter  proofs attracted the attention of Bol \cite{bol}, Klotz \cite{klotz:umbilic},  and Titus \cite{titus} in the ensuing decades. As late as 1993, however, Scherbel \cite{scherbel} was still correcting some errors  in these simplifications, while reconfirming the validity of Hamburger's theorem. Another reexamination of  the proof of the analytic case appears in a comprehensive paper of Ivanov \cite{ivanov} who supplies his own arguments for clarifying various details. All the works mentioned thus far have been primarily concerned with establishing the analytic version of the conjecture of Loewner, which states that the index of the singularities of principal line fields  is at most one, and thus implies Carath\'{e}odory's conjecture via Poincar\'{e}-Hopf index theorem. See Smyth and Xavier \cite{SX1, SX2, SX3} for studies of Loewner's conjecture in the smooth case, and  Lazarovici \cite{lazarovici} for a global result on principal foliations. Another global result is by Feldman \cite{feldman:umbilic}, who showed that generic closed convex surfaces have four umbilics; also see \cite{ghomi&kossowski} for some applications of the $h$-principle to studying homotopy classes of principal lines. A global generalization of Carath\'{e}odory's conjecture is discussed in \cite{gg:umbilic}, and  an interesting analogue of the conjecture for noncompact complete convex hypersurfaces has been studied by Toponogov \cite{Toponogov}.
A number of approaches to proving the Carath\'{e}odory or Loewner conjectures in the smooth case are discussed in \cite{ovsienko&tabachnikov,nikolaev,GK}, and   more references or background may be found in \cite{sotomayor&garcia,gutierrez:umbilic}.

\section{\Mobius Inversions and Parallel Surfaces:\\ Proof of Theorem \ref{thm:1}}\label{sec:1}
The basic idea for proving Theorem \ref{thm:1} is to blow up an umbilic point of $M$ via a \Mobius inversion; however, for the resulting surface to be  an asymptotically constant graph, we first need to transform $M$ to a positively curved surface which is close to a sphere with respect to the Whitney $\C^1$-topology. These preliminary transformations are obtained by taking a \Mobius inversion of an outer parallel surface of $M$ at a sufficiently large distance, as  described below. 
Let us begin by 
recording that the \emph{\Mobius inversion} of any   closed set $A\subset\R^3$  is given by
$$
m(A):=\cl\left\{\frac{p}{\|p\|^2}\,\Big|\, p\in A\setminus\{o\}\right\},
$$
where $\cl$ denotes the closure in $\R^3$ and $o$ is the origin. It is well-known that $m$ \emph{preserves the principal directions} of $\C^2$ surfaces $M\subset\R^3$, i.e.,  a tangent vector $X\in T_pM$ is a principal direction of $M$ if and only if $dm_p(X)$ is a principal direction of $m(M)$ at $m(p)$, where $d$ denotes the differential map. 
This follows for instance from the way the second fundamental form is transformed under a conformal change of the metric in the ambient space, e.g., see \cite[Lemma P.6.1]{hertrich}. 

\begin{lem}\label{lem:decay}
Let $U\subset\R^2$ be an open neighborhood of $o$, and 
$f\colon U\to\R$ be a $\C^1$ function with $f(o)=\|\nabla f(o)\|=0$. Then there exist $r_0>0$, a bounded  open neighborhood $V\subset\R^2$ of $o$, and a  function $\ol f\colon\R^2\setminus V\to\R$, such that 
$$
m\big(\graph(f\big|_{B_{r_0}})\big)=\graph(\ol f).
$$
Furthermore, if
$f$ is $\C^2$, has positive curvature, and  an umbilic at $o$, then $\ol f$ is asymptotically constant.
\end{lem}
\begin{proof}
\textbf{(I)} Since $f$ is $\C^1$, and $\|\nabla f(o)\|=0$, we may choose $r_0$ so small that, in polar coordinates,
\begin{equation}\label{eq:r1r2}
|f(r_1,\theta)-f(r_2,\theta)|<|r_1-r_2|,
\end{equation}
for all $0\leq r_1, r_2<r_0$ and $\theta\in\R/2\pi$. We claim then that $m(\graph f|_{B_{r_0}})$ intersects every vertical line at most once. If this claim holds, then the projection of $m(\graph f|_{B_{r_0}})$ into $\R^2$ defines a closed set $\R^2\setminus V$, where $V$ is some bounded open neighborhood of $o$, and the height  of $m(\graph f|_{B_{r_0}})$  over the $xy$-plane yields the desired function $\ol f\colon \R^2\setminus V\to\R$.
To establish our claim,  let $C$ be the collection of all circles $c\subset\R^3$ which pass through $o$ and are tangent to the $z$-axis. Further note that every vertical line coincides with the \Mobius inversion of a circle $c\in C$. 
So  it suffices to show that each circle $c\in C$ intersect $\graph (f|_{B_{r_0}})$ at most  at one point other than $o$. To see this suppose, towards a contradiction, that there exists some circle $c\in C$ which intersects $\graph (f|_{B_{r_0}})$ at two distinct points $p_1$, $p_2$ other than $o$. Then $p_i=(r_i,\theta_0,f(r_i,\theta_0))$ in cylindrical coordinates for some fixed $\theta_0\in\R/2\pi$. Set $f(r):=f(r,\theta)$. Then \eqref{eq:r1r2} yields that 
$$
|f(r_i)|=|f(r_i)-f(o)|< r_i.
$$
After a rescaling, we may assume that $c$ has radius one. Then,  
$$
|f(r_i)|=\sqrt{1-(1-r_i)^2}=\sqrt{r_i(2-r_i)}.
$$
The last two displayed expressions yield that $r_i>1$. On the other hand, by \eqref{eq:r1r2} and the triangle inequality,
$$
|r_2-r_1|\geq |f(r_2)-f(r_1)|\geq|f(r_2)|-|f(r_1)|\geq \sqrt{r_2(2-r_2)} -r_1.
$$
So if we suppose that $r_2\geq r_1$,  the last expression shows that $r_2<1$ and we have a contradiction.

\textbf{(II)} 
Now suppose that $f$ is $\C^2$,  has an umbilic at  $o$, and is positively curved at $o$. To show that $f$ is asymptotically constant we proceed in three stages: 

\textbf{(II.1)} First we derive some estimates for $f$ and its partial derivatives $f_r$ and $f_\theta$ in polar coordinates. Since principal curvatures of $f$ are equal and positive at $o$, we may assume after a rescaling that  $f-r^2$
vanishes up to order $2$ at $o$. 
So it follows that
$
\lim_{r\to0}(f-r^2)/r^2=0.
$
Consequently the function $\rho$ given by $\rho(o):=0$, and $\rho(x,y):=(f-r^2)/r^2$ for $(x,y)\ne o$ is continuous, and yields our first estimate:
\begin{equation}\label{eq:f}
f=r^2\(1+\frac{f-r^2}{r^2}\)=r^2(1+\rho).
\end{equation}
Similarly, since $f-r^2$ vanishes up to second order at $o$, it follows that  $f_x-2x$ and $f_y-2y$
vanish   up to the first order at $o$, which in turn yields that $f_x=2x+r\xi$ and $f_y=2y+r\eta$, for some continuous functions
$\xi$ and $\eta$ which vanish at $o$.
Using the chain rule, we now compute that
\begin{equation}\label{eq:fr}
f_r=\frac{xf_x+yf_y}{r}=\frac{x(	2x+r\xi)+y(2x+r\eta)}{r}
  =2r+\left( \frac xr \xi+\frac yr \eta\right)r=r(2+\phi),
\end{equation}
where  $\phi(o):=0$ and 
$
\phi(x,y):=(x/r) \xi+(y/r) \eta 
$
for $(x,y)\ne o$.
Since $|x/r|,|y/r|\le 1$, it follows that $|\phi|\le $$|\xi|+|\eta|$,  which shows that $\phi$ is continuous and  vanishes at $o$. Likewise
\begin{equation}\label{eq:ftheta}
f_\theta=-yf_x+xf_y=-y(2x+r\xi)+x(2y+r\eta)
   =\left(-\frac yr \xi +\frac xr\eta\right) r^2=r^2\psi,
\end{equation}
where $\psi$ is again continuous  and vanishes at $o$.

\textbf{(II.2)} Next we show that $\ol f$ converges uniformly to a constant at infinity. With $r_0$ as in part (I) and for every $p\in B_{r_0}$, set
\begin{equation*}\label{eq:mobius}
\big(\ol p, \ol f(\ol p)\big):=m\Big(\big(p,f(p)\big)\Big)=\(\frac{p}{\|p\|^2+f^2(p)},\, \frac{f(p)}{\|p\|^2+f^2(p)}\).
\end{equation*}
Thus if $p=(r,\theta)$ and $\ol p=(\ol r,\ol \theta)$ in polar coordinates, then 
 $\ol\theta=\theta$, and \eqref{eq:f} yields that
\begin{align}\label{eq:olfolr}
\ol r&=\frac{r}{r^2+f^2(r,\theta)}=\frac{1}{r+r^3(1+\rho)^2},\\
  \ol f(\ol r,\theta)&=\frac{f(r,\theta)}{r^2+f^2(r,\theta)}
 = \frac{1+\rho}{1+r^2(1+\rho)^2}.\label{eq:olfolr2}
\end{align}
Choosing  $r_0$ sufficiently small,  we may assume that $r^3(1+\rho)^2<r$ 
for $r<r_0$.  Then \eqref{eq:olfolr} yields
$
1/2r<\ol{r}<1/r.
$
So  $\ol r\to\infty$ if and only if $r\to 0$.
Now by \eqref{eq:olfolr2} 
$$
\lim_{\ol r\to\infty}\ol f(\ol r,\theta)=\lim_{r\to0}\frac{1+\rho}{1+r^2(1+\rho)^2}   =1,
$$
and the continuity of $\rho$  ensures that the convergence is uniform (with respect to
$\theta$).

\textbf{(II.3)} It remains only to check the rate of decay of $\nabla \ol f$. First note that by  \eqref{eq:olfolr} 
\begin{equation}\label{eq:rolr}
\lim_{r\to0}r\ol r=\lim_{\ol r\to\infty}r\ol r=\lim_{\ol r\to\infty}\frac{1}{1+r^2(1+\rho)^2}=1,
\end{equation}
where the continuity of $\rho$ again ensures the uniformity of the convergence. We need to show that 
$\lim_{\ol r\to\infty}\ol r\|\nabla \ol f(\ol r,\theta)\|=0$ uniformly. Thus, since $\nabla \ol f=(\ol f_{\ol r},\ol f_\theta/{\ol r})$ in polar coordinates, it suffices to check that 
\begin{equation*}
\lim_{\ol r\to\infty} \ol r\ol f_{\ol r}(\ol r,\theta)=0, \quad \text{and} \quad \lim_{\ol r\to\infty} \ol f_\theta(\ol r,\theta)=0,
\end{equation*}
uniformly. To establish the first convergence note that, by the chain rule and estimates \eqref{eq:f} and \eqref{eq:fr}
$$
\ol r \ol f_{\ol r}=\ol r\, \frac{\partial \ol f/\partial r}{\partial \ol r/\partial r}=\ol r\,\frac{r^2 f_r-2rf-f^2f_r}{-r^2+f^2-2rff_r}
 = r\ol r  \frac{\phi-2\rho-r^2(1+\rho)^2(2+\phi)}{-1+r^2(1+\rho)^2-2r^2(1+\rho)(2+\phi)}.
$$
Thus by \eqref{eq:rolr}, and since $\rho(o)=\phi(o)=0$,
$$
\lim_{\ol r\to\infty}\ol r\ol f_{\ol r}=\lim_{r\to 0}\frac{\phi-2\rho-r^2(1+\rho)^2(2+\phi)}{-1+r^2(1+\rho)^2-2r^2(1+\rho)(2+\phi)}=0,
$$
where the convergence is  uniform  by the  continuity of $\rho$ and $\phi$.  Similarly, \eqref{eq:olfolr2} together with estimates \eqref{eq:f} and \eqref{eq:ftheta} yield that
$$
\lim_{\ol r\to\infty}\ol f_\theta=\lim_{r\to0}\frac{(r^2-f^2)f_\theta}{(r^2+f^2)^2} 
=\lim_{r\to0} \frac{(1-r^2(1+\rho)^2)\psi}{(1+r^2(1+\rho)^2)^2} =0,
$$
since $\psi(o)=0$, and  the convergence is once again uniform by the continuity of $\rho$ and $\psi$.
\end{proof}

 The next  lemma we need employs the notion of an outer parallel surface, which is defined as follows. A \emph{closed convex surface} $M$ is the boundary of a compact convex set with interior points in $\R^3$. If  $M$ is $\C^{k\geq 1}$, and $n\colon M\to\S^2$ is its outward unit normal vector field or Gauss map, then for any $r\geq 0$, the \emph{outer parallel surface} of $M$ at the distance $r$ is defined as
$$
M^r:=\{\, p+r n(p)\mid p\in M\,\},
$$
which is again a $\C^k$ surface \cite{ghomi:stconvex,krantz&parks:dist}. 
Furthermore it is easy to show that the mapping $f\colon M\to M^r$ given by $f(p):=p+r n(p)$ preserves the principal directions (assuming $M$ is $\C^2$). To see this, let $X\in T_p M$ be a principal direction of $M$, then $dn_p(X)=k(p) X$. So $df_p(X)=(1+r k(p))X$. In particular, $T_p M$ and $T_{f(p)} M^r$ are parallel. Thus $n^r(f(p))=n(p)$ where  $n^r\colon M^r\to\S^2$ is the Gauss map of $M^r$. Now we may compute that 
$$
dn^r_{f(p)}\Big(df_p(X)\Big)=d(n^r\circ f)_p(X)=dn_p(X)=k(p) X=\frac{k(p)}{1+r k(p)}df_p(X),
$$
which shows $f$ preserves principal directions as claimed. Next, 
we also need to recall that the space of $\C^1$ maps $\S^2\to\R^3$ carries  the \emph{Whitney $\C^1$-topology}, which may be defined by stipulating that a pair of $\C^1$ mappings $f, g\colon \S^2\to\R^3$ are \emph{$\C^1$-close} provided that $\|f-g\|\leq \epsilon$ and $\|f_i-g_i\|\leq \epsilon$, $i=1$, $2$, with respect to some atlas of local coordinates on $\S^2$. A pair of embedded spheres are said to be $\C^1$-close if their inclusion maps are $\C^1$-close.

\begin{lem}\label{lem:sphere}
 For every $\C^2$ closed convex surface $M\subset\R^3$, there exists a $\C^2$ closed \emph{positively curved} surface $M'$  which may be arbitrarily $\C^1$-close to $\S^2$,  and a diffeomorphism $M\to M'$ which preserves the principal directions.
 \end{lem}
\begin{proof}
Since $M$ is $\C^2$ it has an inner support ball at each point by Blaschke's rolling theorem \cite{schneider:book}. Indeed, if we let $\rho$ be the minimum of the radii of curvature of $M$, then through each point $p$ of $M$ there passes a  ball $B_p$ of radius $\rho$ which lies inside the convex body bounded by $M$. Suppose that there exists a common point $o$ in the interior of all these inner support balls. Then through each point $m(p)$ of the \Mobius inversion $m(M)$ there passes a ball $m(B_p)$ which contains $m(M)$. Hence $m(M)$ has strictly positive curvature, and is the desired surface $M'$, since $m$ preserves the principal directions. If, on the other hand, the inner support balls of $M$ do not have a common point, we may replace $M$ by an outward parallel surface $M^r$ at a sufficiently large distance $r$. Then the inner support balls of $M^r$  will pass through a common point, and it remains only to recall that that the mapping  $M\to M^r$ given by $p\mapsto p+r n(p)$ preserves the principal directions, as we demonstrated above.  Finally, choosing $r$ sufficiently large, we may make sure that $M^r$ is as $\C^1$-close to a sphere as desired. Indeed $M_r/(1+r)$ converges to a sphere $S$ with respect to the Whitney $\C^1$-topology  as $r\to\infty$, and after a rigid motion and rescaling we may assume that $S=\S^2$.
\end{proof}

We only need one more basic fact for the proof of Theorem \ref{thm:1}. A \emph{round sphere} is any embedding $\S^2\to\R^3$ which is obtained by a rigid motion and homothety of the standard sphere.

\begin{lem}\label{lem:graph}
Let $M\subset\R^3$ be a $\C^1$ embedded sphere which is tangent to the $xy$-plane at $o$. If $M$ is sufficiently $\C^1$-close to a round sphere, then the M\"{o}bius inversion of $M$ forms a graph over the $xy$-plane.
\end{lem}
\begin{proof}
The overall strategy here is similar to that in the proof of Lemma \ref{lem:graph}, i.e., we just need to check that every circle $c\in C$ intersects $M$ at precisely two points, where $C$ again stands for the space of all circles in $\R^3$ which are tangent to the $z$-axis at $o$. To  this end first recall that, as we showed in the proof of Lemma \ref{lem:decay}, there exists an $\epsilon>0$ such that circles  $c\in C$ of radius $\leq\epsilon$ intersect $M$ at precisely two points already.
So it suffices to consider only the collection of circles $C_{\geq\epsilon/2}\subset C$ of radius $\geq \epsilon/2$. 
Next note that
by assumption $M$ is $\C^1$-close to a sphere $S$, which we may assume to be tangent to the $xy$-plane at $o$.   Further every circle $c\in C$ intersects $S$ orthogonally. Since $M$ and $S$ are $\C^1$-close, we may conclude then that $M$ intersects every circle $c\in C_{\geq\epsilon/2}$ transversely. Thus it follows that the number $\#(c\cap M)$ of intersection points of $c$ with $M$ is finite, and the mapping $c\mapsto \#(c\cap M)$ is locally constant on $C_{\geq\epsilon/2}$ with respect to the Hausdorff topology on $C$. This yields that $c\mapsto\#(c\cap M)$ is a constant function on  $C_{\geq\epsilon/2}$, since $C_{\geq\epsilon/2}$ is  connected.
 So $\#(c\cap M)=2$ for every $c\in C_{\geq\epsilon/2}$, since circles of radius $\epsilon$ intersects $M$ precisely twice.
\end{proof}

Now we have all the pieces in place to complete the proof of Theorem \ref{thm:1}, as follows.
After a rigid motion we may assume that an umbilical point $p$ of $M$ lies at the origin $o$ of $\R^3$, and $M$ is tangent to the $xy$-plane at $o$. Further, by Lemma \ref{lem:sphere} we may assume that $M$  is $\C^1$-close to a sphere and is positively curved. So the M\"{o}bius inversion $m(M)$ forms a graph over the $xy$-plane by Lemma \ref{lem:graph}. Finally, this graph is asymptotically constant by Lemma \ref{lem:decay}.

\begin{note}
Proof of Theorem \ref{thm:1} in this section reveals that the theorem is valid not just for convex surfaces, but applies to any $\C^2$ 
closed surface $M\subset\R^3$ which may be transformed to a convex one by means of \Mobius transformations, and the operation of moving $M$ parallel to itself.
\end{note}

\begin{note}
In Lemma \ref{lem:sphere} we may choose the surface $M'$ to be 
$\C^2$-close to $M$. Indeed, as we showed in the proof of the lemma, we may assume that $M$ 
has positive curvature. Then the \emph{support function} $h_M\colon\R^2\to\R$  of $M$ is also $\C^2$, e.g., see \cite{schneider:book}. 
Further, the support function of the outer parallel surface $M^r$  is $r + h_M$, so if we dilate $M^r$
by the factor  $1/r$, then  
$$
h_{\frac{1}{r} M^r} = 1 +\frac{h_M}{r} 
$$
which converges to $1$ in the $\C^2$ 
topology. Thus $(1/r)M^r$
converges to 
$\S^2$ in the $\C^2$ topology. This argument also shows that if $M$ is $\C^k$, $2 
\leq k \leq\infty$, then $(1/r) M^r$ 
converges 
to $\S^2$ in the $\C^k$ 
topology. 
\end{note}

\section{Applications of the Divergence Theorem:\\Proofs of Theorems \ref{thm:2} and \ref{thm:3}}\label{sec:2}
Let us say that a function $f\colon\R^2\to\R$ \emph{decays uniformly faster} than $1/r$, provided that, in polar coordinates, $\lim_{r\to\infty} rf(r,\theta)=0$ uniformly
 in  $\theta$.  Here $\nabla\cdot$ stands for divergence.

\begin{lem}\label{lem:uv}
Let $V\colon\R^2\to\R^2$ be a   $\C^1$ vector field  whose norm decays uniformly faster than $1/r$. Then 
$$
\lim_{r\to\infty}\int_{B_r}\nabla\cdot V=0.
$$
\end{lem}
\begin{proof}
Let $c(\theta):=(r\cos(\theta), r\sin(\theta))$ be a parametrization for $\partial B_r$, and $n$ be its outward unit normal. By the divergence theorem and Cauchy-Schwartz inequality
\begin{align*}
\left|\int_{B_r}\nabla\cdot V \right|
 &= \left|\int_{\partial B_r} V\cdot n\right| \\
& \leq \int_{\partial B_r}\|V\|=\int_0^{2\pi} \|V(c(\theta))\|\|c'(\theta)\|d\theta=\int_0^{2\pi} \|V(r,\theta)\|r\,d\theta.
\end{align*}
Thus it follows that
$$
\lim_{r\to\infty}\left|\int_{B_r}\nabla\cdot V\right|\leq\lim_{r\to\infty}\int_0^{2\pi} \|V(r,\theta)\|rd\theta=\int_0^{2\pi}\lim_{r\to\infty} \|V(r,\theta)\|rd\theta=0,
$$ 
where exchanging the order of integration and the limit here is warranted by the  assumption that the convergence of the functions $\|V(r,\theta)\|r$ is uniform in $\theta$.
\end{proof}

Equipped with the last fact, we are now ready to complete the proofs of our main results:

\begin{proof}[Proof of Theorem \ref{thm:2}]
Set $X=(X^1, X^2)$, $Y=(Y^1, Y^2)$, and note that $f_{XY}=\sum_{i,j=1}^2 f_{ij}X^i Y^j=f_{YX}$. Thus, if $k_X(p):=k(p,X)$,  then using \eqref{eq:k} we have
 \begin{align*}
(k_X-k_Y)&(1+f_X^2)(1+f_Y^2)\sqrt{1+\|\nabla f(p)\|^2}\\&=f_{XX}(1+f_Y^2)-f_{YY}(1+f_X^2)\\
&=\Big(f_X\big(1+f_Y^2\big)\Big)_X-\Big(f_Y\big(1+f_X^2\big)\Big)_Y.
\end{align*}
Setting $u:=f_X(1+f_Y^2)$, $v:=f_Y(1+f_X^2)$,   we continue the above computation as:
\begin{align*}
u_X-v_Y&=u_1X^1+u_2X^2-v_1Y^1-v_2Y^2\\
&=\big(uX^1-vY^1\big)_1+\big(uX^2-vY^2\big)_2.
\end{align*}
The last line is the divergence  of the vector field 
$$(uX^1-vY^1,uX^2-vY^2),$$
whose norm vanishes uniformly faster than $1/r$ since  $f$ is asymptotically constant. 
So applying Lemma \ref{lem:uv} to this vector field  completes the proof.
 \end{proof}

\begin{proof}[Proof of Theorem \ref{thm:3}]
After a rotation of the coordinate axis we may assume that $\theta_0=0$. A simple computation shows that
$$
k_{X(\theta)}=\frac{ f_{11}\cos ^2(\theta)+f_{12}\sin (2 \theta) +f_{22}\sin ^2(\theta) }{1+ \big(f_{1}\cos (\theta)+ f_{2}\sin (\theta)\big)^2}\frac{1}{\sqrt{1+\|\nabla f\|^2}},
$$
and another straight forward calculation yields
$$
\frac{\partial}{\partial\theta}\Big|_{\theta=0}k_{X(\theta)}=\frac{2\Big( (1+f_1^2) f_{12}- f_2 f_1
   f_{11}\Big)}{\left(1+f_1^2\right)^2}\frac{1}{\sqrt{1+\|\nabla f\|^2}}.
$$
So $(1,0)=X(0)$ is a principal direction of $f$ at $p$ if and only if the following equation holds at $p$
$$
 (1+f_1^2) f_{12}- f_2 f_1
   f_{11}=0,
$$
which is equivalent to
\begin{equation}
\( \frac{f_2}{\sqrt{1+f_1^2}}\)_1=0.
\end{equation}
Thus applying Lemma \ref{lem:uv} to the vector field $(f_2/\sqrt{1+f_1^2},0)$  completes the proof.
\end{proof}

\begin{note}
Proofs of Theorems \ref{thm:2} and \ref{thm:3} do not use the assumption that $f$ converges to a constant at infinity. Indeed these results require only that the gradient vanishes uniformly faster than $1/r$. So $f$ needs not even be bounded (consider for instance any smooth function $f\colon\R^2\to\R$ which coincides with $\ln(\ln r)$ outside a compact set.)
\end{note}

 \begin{note}\label{note:pdes}
  A surface $M\subset\R^3$ has an umbilical point $p$ if and only if its second fundamental form is a multiple of its first fundamental form at $p$. When $M$ is the graph of a function $f\colon\R^2\to\R$, this means that 
  $$
\rank  \left(
\begin{array}{ccc}
f_{11} & f_{12} & f_{22}\\
1+f_1^2 &f_1 f_2 &1+ f_2^2
\end{array}\right)\leq 1,
  $$
 which is equivalent to the following system of equations
\begin{align}\label{eq:s1}
(1+f_1^2) f_{12}- f_1f_2 f_{11}&=0,\\ \notag(1+f_1^2) f_{22}-(1+f_2^2)f_{11}&=0.
\end{align}
Indeed  the first equation holds if and only if $(1,0)$ is a principal direction of $f$, as we showed in the proof of Theorem \ref{thm:3}, and  the second equation holds if and only if  the normal  curvatures of $f$ in the directions $(1,0)$ and $(0,1)$ agree, as we saw in the proof of Theorem \ref{thm:2}.
Also recall that the above system of equations is equivalent to
\begin{align*}
\( \frac{f_2}{\sqrt{1+f_1^2}}\)_1&=0,\\ \Big((1+f_1^2)f_{2} \Big)_2-\Big((1+f_2^2)f_{1}\Big)_1&=0.
\end{align*}
Yet another way to characterize the umbilics of a graph is as the solutions of
$$
\Big(f_{22} \left(1+f_1^2\right)-2 f_{1} f_2
   f_{12}+ f_{11}\left(1+f_2^2\right)\Big)^2-4
   \left(1+f_1^2+f_2^2\right) \left(f_{22}
   f_{11}-f_{12}^2\right)=0.
   $$
The above PDE is obtained by setting $H^2-K$ equal to zero, where $H$ and $K$ are the mean and Gauss curvatures of the graph of $f$ respectively. This also yields a coordinate free expression which is equivalent to the equations above:
$$
\(\nabla\cdot\(\frac{\nabla f}{\sqrt{1+\|\nabla f\|^2}}\)\)^2-\frac{4\det\Hess(f)}{(1+\|\nabla f\|^2)^2}=0.
$$
To prove Carath\'{e}odory's conjecture it suffices to show that any one of the above systems holds at some point, assuming $f$ is asymptotically constant. 
\end{note}

\section{Some Examples}\label{sec:3}

\subsection{}
The graph of a function $f\colon\R^2\to\R$ which is not asymptotically constant may not have any umbilics.
Indeed, any negatively curved graph such as $f(x,y)=xy$, or any cylindrical graph given by $f(x,y)=g(x)$ where $g\colon\R\to\R$ is a smooth function without inflection points are umbilic free. There are even umbilic free graphs which are bounded above and below. One such example is due to Bates \cite{bates}, and another  is given by
$$
f_\lambda(x,y)=1+\lambda\frac{x+y^2}{\sqrt{1+(x+y^2)^2}},
$$
for $\lambda>0$. Indeed, with the aid of a computer algebra system, one may easily verify that the equations \eqref{eq:s1} are never simultaneously satisfied for $f_\lambda$, unless $\lambda=0$. The  \Mobius inversion of  $\graph(f_\lambda)$  is $\C^\infty$  in the complement of one point, is differentiable everywhere, and converges
to a sphere with respect to Hausdorff distance as $\lambda\to 0$; see Figure \ref{fig:1}. The differentiability of the inversion is due to the fact that $\graph(f_\lambda)$ is contained between a pair of horizontal planes, which in turn implies that the inversion rests between a pair of spheres at $o$. In particular, the tangent cone of the inversion at $o$ is a plane.
\begin{figure}[h] 
   \centering
   \includegraphics[height=1.2in]{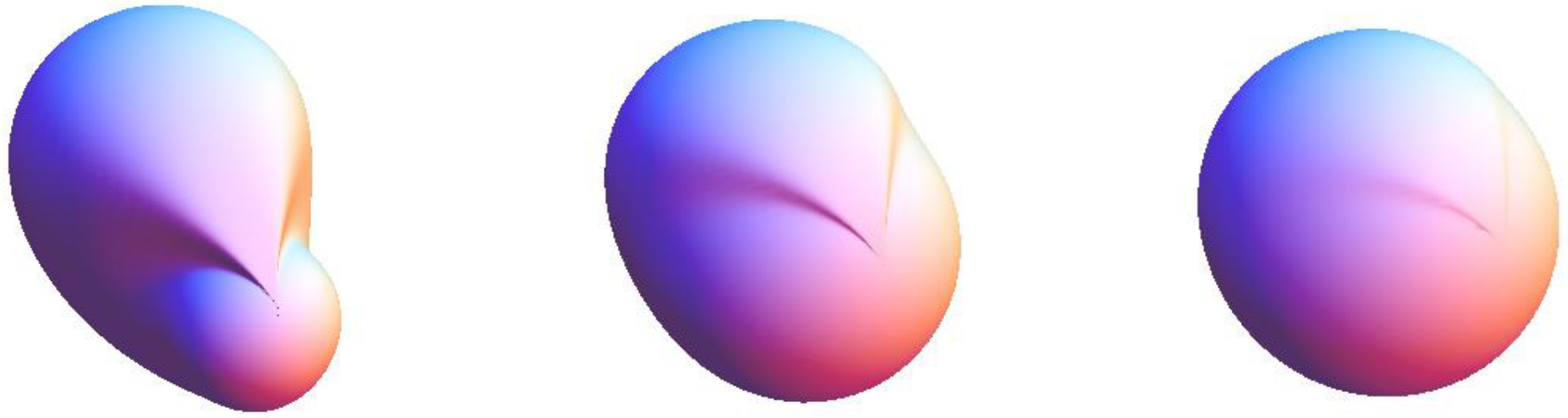} 
   \caption{}
   \label{fig:1}
\end{figure}

\subsection{}
Using \Mobius inversions, one may construct closed \emph{convex} surfaces which are umbilic free and $\C^\infty$ in the complement of one point. A family of such examples is given by  inverting the graph of
$$
f_\lambda(x,y)=1+\lambda\sqrt{1+x^2},
$$
for  sufficiently small $\lambda>0$.  It is easy to see that these surfaces are umbilic free for all $\lambda\neq 0$. Further, if $\lambda$ is small, then each support plane of these graphs separates the graph from the origin $o$. Thus the inversion of the graph has a supporting sphere at each point which shows that its must be convex. Furthermore, since $f_\lambda$ converges to the plane $z=1$, it follows that 
inversions of $\graph(f_\lambda)$ converge to a sphere as $\lambda\to 0$, as shown in Figure \ref{fig:2}. 
\begin{figure}[h] 
   \centering
   \includegraphics[height=1.2in]{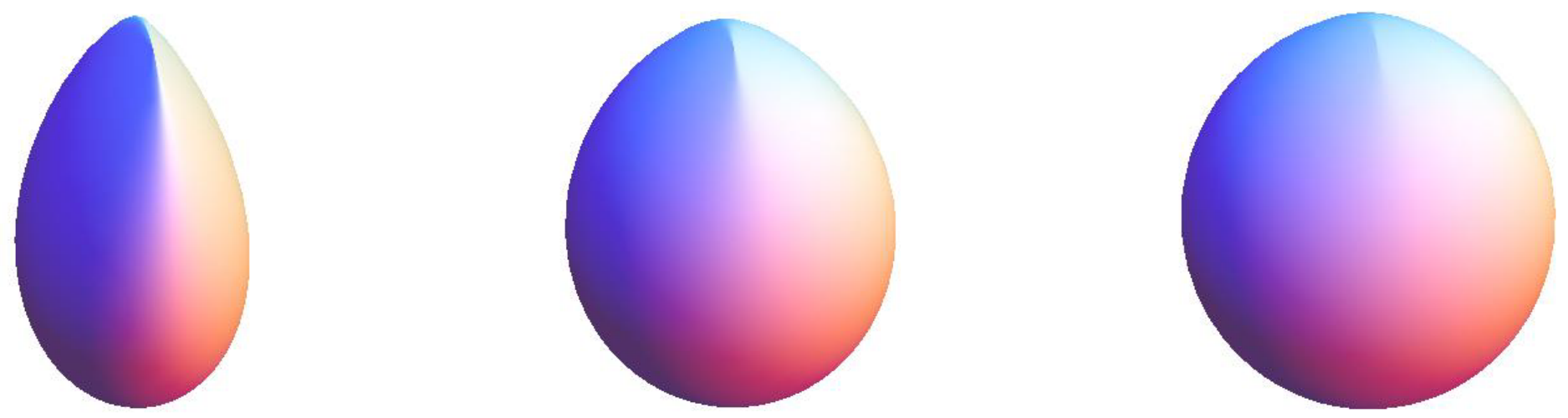} 
   \caption{}
   \label{fig:2}
\end{figure}
An approximation of these inversions, which may be achieved for instance by applying a certain convolution to their support functions \cite{schneider:book,ghomi:polytopes}, yields $\C^\infty$ closed convex surfaces all of whose umbilics are contained in a region with arbitrarily small diameter and total curvature, c.f. \cite{gutierrez&sanchez}. Indeed, a neighborhood of the singular point may be represented as the graph of a convex function over a support plane. Identifying the support plane with the $\R^2\times\{0\}\subset\R^3$, we then obtain a convex function $f\colon B_r\to\R$ which is $\C^\infty$ and positively curved in the complement of the origin $o$. For every open neighborhood $U$ of  $o$, with closure $\ol U\subset B_r$, there exists a $\C^\infty$ convex function $\tilde f\colon B_r\to\R$ such that $\tilde f= f$ on $B_r-U$, see  \cite{ghomi:convexfunctions}. Replacing $\graph(f)$ with $\graph(\tilde f)$ then yields a smoothing of our surfaces, which preserves each surface in the complement of any given open neighborhood of the singularity.

\subsection{}
Another example of a family of convex surfaces which are umbilic free and $\C^\infty$ in the complement of one point is obtained by inverting the graphs of
$$
f_\lambda(x,y)=1+\lambda\(\sqrt{1+x^2}+x+\sqrt{1+y^2}+y\)
$$
for $\lambda>0$. It is easy to show that these graphs never satisfy the first equation in \eqref{eq:s1} unless $\lambda=0$. It is also worth remarking that, in contrast to the previous example, these graphs have everywhere positive curvature. 
Thus one may say that Carath\'{e}odory's conjecture has no analogue for complete convex surfaces which are not compact. Similar to the previous example, the inversion of these graphs  converge to a sphere as $\lambda\to 0$, see Figure \ref{fig:3}; 
\begin{figure}[h] 
   \centering
   \includegraphics[height=1.1in]{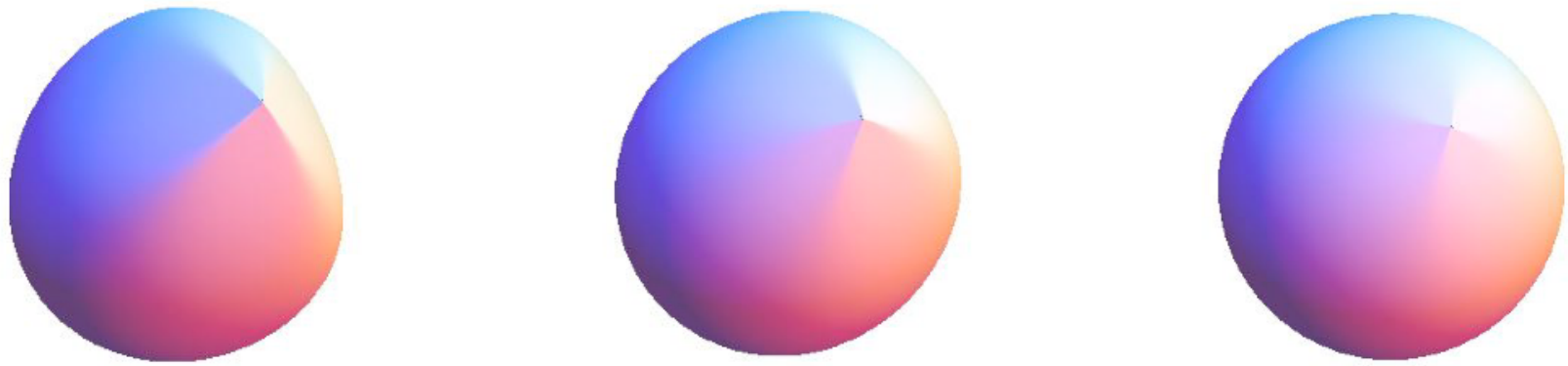} 
   \caption{}
   \label{fig:3}
\end{figure}
however, here, the singularity is of the ``cone type", i.e., the normal cone at the singularity has dimension $3$ (as opposed to the previous example where the singularity was of the ``ridge type", i.e., the normal cone was only $2$-dimensional). 
Similar examples may be generated by any function of the form $f_\lambda(x,y)=1+\lambda(g(x)+h(y))$ where
$g$ and $h$ are $\C^2$, neither $g'$ or $h'$ ever vanishes, and $g''$ and $h''$ are
always positive.

\section*{Acknowledgements}
We thank Jason Cantarella,  Serge Tabachnikov, Brian White, and Frederico Xavier for useful communications.

\bibliographystyle{abbrv}
\bibliography{references}

\end{document}